\documentclass[reqno,12pt]{amsart} 
\numberwithin{equation}{section}
\usepackage{amssymb}
\usepackage[mathscr]{eucal}
\usepackage{amsmath}
\usepackage{latexsym}
\usepackage{amsthm}
\theoremstyle{plain} 
\newtheorem{theorem}{\indent\sc Theorem}[section]
\newtheorem{lemma}[theorem]{\indent\sc Lemma}

\theoremstyle{definition} 
\newtheorem{definition}[theorem]{\indent\sc Definition}
\newtheorem{remark}[theorem]{\indent\sc Remark}

\begin{document}
\setcounter{page}{1}
\title[Weighted Norm Inequalities for One-sided Oscillatory Integral Operators]
{Weighted Norm Inequalities for One-sided Oscillatory Integral Operators$^*$}
\author[Zunwei Fu ]{Zunwei Fu}
\author[Shaoguang Shi]{Shaoguang Shi} 
\author[Shanzhen Lu]{Shanzhen Lu} 
\thanks{
\endgraf
}

%
\thanks{ 
$^{*}$This work was partially supported by
 NSF of China (Grant Nos. 10871024, 10901076 and 10931001), NSF of Shandong Province
 (Grant No. Q2008A01) and the Key Laboratory of Mathematics and Complex System (Beijing Normal
University), Ministry of Education, China.}

\address{
School of Sciences\endgraf
Linyi
University \endgraf
Linyi 276005\endgraf
P. R. China}
\email{lyfzw@tom.com}

\address{
School of Mathematical Sciences\endgraf
Beijing Normal
University\endgraf
Beijing 100875\endgraf
and\endgraf
School of Sciences\endgraf
Linyi
University \endgraf
Linyi 276005\endgraf
P. R. China}
\email{shishaoguang@yahoo.com.cn}

\address{
School of Mathematical Sciences\endgraf
Beijing Normal
University\endgraf
Beijing 100875\endgraf
P. R. China}
\email{lusz@bnu.edu.cn}


\maketitle
\begin{abstract}
The purpose of this paper is to establish the weighted norm inequalities of one-sided oscillatory
integral operators by the aid of interpolation of operators with change of measures.
\end{abstract}

\section*{Introduction} 
Many operators in harmonic analysis or partial differential equation
are related to some versions of oscillatory integrals, such as the
Radon transform which has important applications in the CT technology.
Among numerous papers dealing with norm inequalities of integral operators in some function spaces, we refer to[2], [3], [9], [14] and [15].
More general, let us now consider a class of oscillatory integrals
defined by Ricci and Stein [10]:
$$Tf(x)=\mathrm{p.v.}\int_{\mathbb{R}}e^{iP(x,y)}K(x-y)f(y)dy,$$
where $P(x,y)$ is a real valued polynomial defined on $\mathbb{R}\times\mathbb{R}$, and $K$ is a standard Calder\'{o}n-Zygmund Kernel.
That means $K$ satisfies
$$|K(x)|\leq\frac{C}{|x|},\, x\neq 0,\eqno(1.1)$$
and
$$|K(x-y)-K(x)|\leq\frac{C|y|}{|x|^{2}},\, x\neq y.\eqno(1.2)$$

We recover the Ricci and Stein's celebrated result [10] on oscillatory integrals as follows.

\begin{theorem} Suppose $K(x,y)$ satisfies $\mathrm{(1.1)}$ and $\mathrm{(1.2)}$. If the Calder\'{o}n-Zygmund singular integral operator
$$\widetilde{T}f(x)=\mathrm{p.v.}\int_{\mathbb{R}}K(x-y)f(y)dy$$
is of type $(L^{2}, L^{2})$, then for any real polynomial $P(x,y)$, the oscillatory integral operator
$T$is of type $(L^{p},L^{p})$, $1<p<\infty$, where its norm depends only on the total degree of $P$, but not on the coefficients of $P$.
\end{theorem}

The study of one-sided operators was motivated not only
as the generalization of the theory of both-sided ones but also their
natural appearance in harmonic analysis, such as the one-sided Hardy-Littlewood maximal operator
$$M^{+}f(x)=\sup_{h>0}\frac{1}{h}\int_{x}^{x+h}|f(y)|dy$$
arising in the ergodic  maximal function. The one-sided weight $A_{p}^{+}$ classes were introduced by Sawyer [12], i.e, there exists a constant $C$ such that for all real $a$ and positive $h$:
$$\left(\frac{1}{h}\int_{a-h}^{a}w(x)dx\right)\left(\frac{1}{h}\int_{a}^{a+h}w(x)^{1-p'}dx\right)^{p-1}\leq C,$$
where $1<p<+\infty, 1/p+1/p'=1$. The smallest constant for which this is satisfied will be called the $A_{p}^{+}$ constant of $w$ and will be denoted by $A_{p}^{+}(w)$.

The counterpart of $M^{+}$ is defined as
$$M^{-}f(x)=\sup_{h>0}\frac{1}{h}\int_{x-h}^{x}|f(y)|dy.$$
The weight $w\in A_{p}^{-}$ means
$$\left(\frac{1}{h}\int_{a}^{a+h}w(x)dx\right)\left(\frac{1}{h}\int_{a-h}^{a}w(x)^{1-p'}dx\right)^{p-1}\leq C$$
for all real $a$ and positive $h$. The smallest constant for which this is satisfied will be called the $A_{p}^{-}$ constant of $w$ and will be denoted by $A_{p}^{-}(w)$.
\begin{remark}
The general definition of $A_{p}^{+}(A_{p}^{-})$ was introduced in
[7] as follows:
$$
A_{p}^{+}:\sup_{a<b<c}\frac{1}{(c-a)^{p}}\int_{a}^{b}w(x)dx\left(\int_{b}^{c}w(x)^{1-p'}dx\right)^{p-1}\leq C,
$$
and $$
A_{p}^{-}:\sup_{a<b<c}\frac{1}{(c-a)^{p}}\int_{b}^{c}w(x)dx\left(\int_{a}^{b}w(x)^{1-p'}dx\right)^{p-1}\leq C.
$$
\end{remark}

It is easy to see that $A_{p}\subset A_{p}^{+}$, $A_{p}\subset A_{p}^{-}$ and $A_{p}=A_{p}^{+}\bigcap A_{p}^{-}$,  where $A_{p}$ denotes the Muckenhoupt classes:
$$\left(\frac{1}{|I|}\int_{I}w(x)dx\right)\left(\frac{1}{|I|}\int_{I}w(x)^{1-p'}dx\right)^{p-1}\leq C.$$
Here $I$ denotes any intervals in $\mathbb{R}$. $A_{p}$ class on $\mathbb{R}^{n}$ can be naturally defined.

\begin{theorem}$[12]$\,\, Let $1<p<\infty$. Then\\
$\mathrm{(1)}$\,\, $M^{+}$ is bounded in $L^{p}(w)$ if and only if $w\in A_{p}^{+}$.\\
$\mathrm{(2)}$\,\, $M^{-}$ is bounded in $L^{p}(w)$ if and only if $w\in A_{p}^{-}$.
\end{theorem}

We say that $w$ satisfies the $A_{1}^{+}(A_{1}^{-})$ condition if $M^{-}w(M^{+}w)\leq Cw$. The smallest such constant $C$ will be called the $A_{1}^{+}(A_{1}^{-})$ constant of $w$ and will be denoted by $A_{1}^{+}(w)(A_{1}^{-}(w))$. By Lebesgue's differentiation Theorem, we can easily prove $A_{1}^{+}(w) (A_{1}^{-}(w))\geq1$. In [8] the class $A_{\infty}^{+}$ was introduced as $A_{\infty}^{+}=\bigcup_{p<\infty}A_{p}^{+}$. These classes are of interest, not only because they control the boundedness of the one-sided Hardy-Littlewood maximal operator, but they are the right classes for the weighted estimates for one-sided Calder\'{o}n-Zygmund singular integrals which are defined by

$$\widetilde{T}^{+}f(x)=\lim_{\varepsilon\rightarrow0^{+}}\int_{x+\varepsilon}^{\infty}K(x-y)f(y)dy$$
and
$$
\widetilde{T}^{-}f(x)=\lim_{\varepsilon\rightarrow0^{+}}\int_{-\infty}^{x-\varepsilon}K(x-y)f(y)dy,
$$
where $K$ is a standard Calder\'{o}n-Zygmund kernel with support in $\mathbb{R}^{-}=(-\infty, 0)$ and $\mathbb{R}^{+}=(0, +\infty)$, respectively.

\begin{theorem} $[1]$\,\, Let $1<p<\infty$. Suppose $K$ satisfies $\mathrm{(1.1)}$, $\mathrm{(1.2)}$ and satisfies $$\left|\int_{\varepsilon<|x|<N}K(x)dx\right|\leq C$$
for all $\varepsilon$ and all $N$, with $0<\varepsilon<N$, and furthermore $\lim_{\varepsilon\rightarrow0^{+}}\int_{\varepsilon<|x|<N}K(x)dx$ exists. Then\\
$\mathrm{(1)}$\,\, $\widetilde{T}^{+}$ is bounded in $L^{p}(w)$ if and only if $w\in A_{p}^{+}$.\\
$\mathrm{(2)}$\,\, $\widetilde{T}^{-}$ is bounded in $L^{p}(w)$ if and only if $w\in A_{p}^{-}$.
\end{theorem}

The above result is the one-sided version of weighted norm inequality of singular integral due to Coiffman and Fefferman [2].

In 1992, Lu and Zhang [5] gave the weighted result of Theorem 0.1.

\begin{theorem} Suppose $K(x,y)$ satisfies $\mathrm{(1.1)}$ and $\mathrm{(1.2)}$. If the operator
$\widetilde{T}$
is of type $(L^{2},L^{2})$, then for any real polynomial $P(x,y)$, the oscillatory integrals operator
$T$
is of type $(L^{p}(w),L^{p}(w))$, $w\in A_{p}$ and $1<p<\infty$. Here its norm depends only on the total degree of $P$ and $A_{p}(w)$, but not on the coefficients of $P$.
\end{theorem}

Inspired by [1] and [5], we will study the one-sided version of
Theorem 0.5 by the aid of interpolation of operators with change of
measures and the weak reverse H\"{o}lder inequality. Throughout this paper the letter $C$ will denote a
positive constant which may vary from line to line but will remain
independent of the relevant quantities.


\section{Main Results}
We first give the definition of one-sided oscillatory integral operator $T^{+}(T^{-})$:
$$\begin{array}{rl}
\displaystyle T^{+}f(x)&=\displaystyle\lim_{\varepsilon\rightarrow0^{+}}\int_{x+\varepsilon}^{\infty}e^{iP(x,y)}K(x-y)f(y)dy\\&=\displaystyle \mathrm{p.v.}\int_{x}^{\infty}e^{iP(x,y)}K(x-y)f(y)dy
\end{array}$$
and
$$\begin{array}{rl}
\displaystyle
T^{-}f(x)&=\displaystyle\lim_{\varepsilon\rightarrow0^{+}}\int_{-\infty}^{x-\varepsilon}e^{iP(x,y)}K(x-y)f(y)dy\\&=\displaystyle \mathrm{p.v.}\int_{-\infty}^{x}e^{iP(x,y)}K(x-y)f(y)dy,
\end{array}$$
where $P(x,y)$ is a real polynomial defined on $\mathbb{R}\times\mathbb{R}$, and Kernel $K$ is a standard Calder\'{o}n-Zygmund kernel with support in $\mathbb{R}^{-}=(-\infty,0)$ and $\mathbb{R}^{+}=(0,+\infty)$, respectively.

Now, we may state our results as follows:

\begin{theorem} Suppose Kernel $K$ satisfies $\mathrm{(1.1)}$ and $\mathrm{(1.2)}$.\\
$\mathrm{(1)}$\,\,
If the operator
$\widetilde{T}^{+}$
is of type $(L^{2}, L^{2})$, then for any real polynomial $P(x,y)$, the oscillatory integrals operator
$T^{+}$
is of type $(L^{p}(w), L^{p}(w))$ for $w\in A^{+}_{p}$ , $1<p<\infty$.\\
$\mathrm{(2)}$\,\,
If the operator
$\widetilde{T}^{-}$
is of type $(L^{2}, L^{2})$, then for any real polynomial $P(x,y)$, the oscillatory integrals operator
$T^{-}$
is of type $(L^{p}(w), L^{p}(w))$ for $w\in A^{-}_{p}$ , $1<p<\infty$.

Here their norms depend only on the total degree of $P$, $A_{p}^{+}(w)$ and $A_{p}^{-}(w)$, but not on the coefficients of $P$.
\end{theorem}

The rest of this paper is devoted to the argument for Theorem 1.1.
Section 2 contains some preliminaries which are essential to our
proof. In Section 3, we prove Theorem 1.1, this part is partially
motivated by [4] and [5].

\section{Preliminaries}

\begin{lemma}\label{Lemma} $[11], [12]$\,\,
Let $1<p<\infty$, and $w\geq 0$ be locally integrable. Then the following statements are equivalent\\
$\mathrm{(1)}$\,\, $w\in A_{p}^{+}$.\\
$\mathrm{(2)}$\,\,  $w^{1-p'}\in A_{p'}^{-}$. \\
$\mathrm{(3)}$\,\,  There exist $w_{1}\in A_{1}^{+}$ and $w_{2}\in A_{1}^{-}$ such that $w=w_{1}(w_{2})^{1-p}$.
\end{lemma}

According to the definition of $A_{p}^{+}$, we can easily obtain

\begin{lemma}\label{Lemma}
Let $1<p<\infty$ and $w\in A_{p}^{+}$. Then $ A_{p}^{+}(\delta^{\lambda}(w))=A_{p}^{+}(w)$,
where $\delta^{\lambda}(w)(x)=w(\lambda x)$ for all $\lambda>0$.
\end{lemma}

\begin{proof} For $1<p<\infty$, if $w\in A_{p}^{+}$, then
$$\sup_{a<b<c}\frac{1}{(c-a)^{p}}\int_{a}^{b}w(x)dx\left(\int_{b}^{c}w(x)^{1-p'}dx\right)^{p-1}\leq C.$$
For $\lambda>0$, $a'=\lambda a$, $b'=\lambda b$, $c'=\lambda c$ and $d'=\lambda d$, we have

\begin{eqnarray*}
&&\frac{1}{(c-a)^{p}}\int_{a}^{b}w(\lambda x)dx\left(\int_{b}^{c}w(\lambda x)^{1-p'}dx\right)^{p-1}\\
&=& \frac{1}{(c-a)^{p}}\int_{a\lambda}^{b\lambda}w(x)\lambda ^{-1}dx\left(\int_{b\lambda}^{c\lambda}w(x)^{1-p'}\lambda ^{-1}dx\right)^{p-1}\\
&=& \frac{1}{\left(\lambda(c-a)\right)^{p}}\int_{a\lambda}^{b\lambda}w(x)dx\left(\int_{b\lambda}^{c\lambda}w(x)^{1-p'}dx\right)^{p-1}\\
&=& \frac{1}{(c'-a')^{p}}\int_{a'}^{b'}w(x)dx\left(\int_{b'}^{c'}w(x)^{1-p'}dx\right)^{p-1}\\
&\leq& C.\end{eqnarray*}

The proof is complete.
\end{proof}

\begin{definition} $[11]$\,\,
A weight $w$ satisfies the one-sided reverse H\"{o}lder $RH_{r}^{+}$ condition, if there exists $C>0$ such that for any $a<b$
$$\int_{a}^{b}w(x)^{r}dx\leq C(M(w\chi_{(a,b)})(b))^{r-1}\int_{a}^{b}w(x)dx,\eqno(2.1)$$
where $1<r<\infty$.
\end{definition}

The smallest such constant will be called the $RH_{r}^{+}$ constant of $w$ and will be denoted by $RH_{r}^{+}(w)$. Corresponding to
classical reverse H\"{o}lder inequality, (2.1) is also named weak reverse H\"{o}lder inequality.

\begin{definition} $[11]$\,\,
A weight $w$ satisfies the one-sided reverse H\"{o}lder $RH_{\infty}^{+}$ condition, if there exists $C>0$ such that
$w(x)\leq Cm^{+}w(x)$ for almost all $x\in \mathbb{R}$, where $m^{+}$ is the one-sided minimal operator which defined as
$$m^{+}f(x)=\inf_{h>0}\frac{1}{h}\int_{x}^{x+h}|f|dy.$$
\end{definition}

The smallest such constant will be called the $RH_{\infty}^{+}$ constant of $w$ and will be denoted by $RH_{\infty}^{+}(w)$, it is clear that $RH_{\infty}^{+}(w)\geq 1$.

The following lemma gives several characterizations of $RH_{r}^{+}$ where the constants are not necessary the same.

\begin{lemma} $[11]$\,\,
Let $a<b<c<d$, $1<r<\infty$, and $w\geq 0$ be locally integrable. Then the following statements are equivalent\\
$\mathrm{(1)}$\,\, $\int_{a}^{b}w(x)^{r}dx\leq C(M(w\chi_{(a,b)})(b))^{r-1}\int_{a}^{b}w(x)dx$.\\
$\mathrm{(2)}$\,\, $\frac{1}{b-a}\int_{a}^{b}w(x)^{r}dx\leq C(\frac{1}{c-b}\int_{b}^{c}w(x)dx)^{r}$ with $b-a=2(c-b)$.\\
$\mathrm{(3)}$\,\, $\frac{1}{b-a}\int_{a}^{b}w(x)^{r}dx\leq C(\frac{1}{d-c}\int_{c}^{d}w(x)dx)^{r}$ with $b-a=d-b=2(d-c)$.\\
$\mathrm{(4)}$\,\, $\frac{1}{b-a}\int_{a}^{b}w(x)^{r}dx\leq C(\frac{1}{c-b}\int_{b}^{c}w(x)dx)^{r}$ with $b-a=c-b$.\\
$\mathrm{(5)}$\,\, $\frac{1}{b-a}\int_{a}^{b}w(x)^{r}dx\leq C(\frac{1}{d-c}\int_{c}^{d}w(x)dx)^{r}$ with $b-a=d-c=\gamma(d-a), 0<\gamma\leq\frac{1}{2}$.
\end{lemma}

\begin{lemma} $[11]$\,\,
A weight $w\in A_{p}^{+}$, for $p>1$ if, and only if there exist
$0<\gamma<\frac{1}{2}$ and a constant $C_{\gamma}$ such that
$b-a=d-c=\gamma(d-a)$ for any $a<b<c<d$, then
$$
\int_{a}^{b}w(x)dx\left(\int_{c}^{d}w(x)^{1-p'}dx\right)^{p-1}\leq C_{\gamma}(b-a)^{p}.
$$
\end{lemma}

Combining the results in [1], [6], [7] and [11], we can deduce Lemma 2.7.  In what follows, we will include its proof with slight modifications for the sake of completeness.
\begin{lemma}
Let $w\in A_{p}^{+}$. Then there exists $\varepsilon>0$ such that $w^{1+\varepsilon}\in A_{p}^{+}$.
\end{lemma}

\begin{proof}
Let $w\in A_{p}^{+}$. By Lemma 2.1, $w=w_{1}w_{2}^{1-p}$ with $w_{1}\in A_{1}^{+},w_{2}\in A_{1}^{-}$. Next, we claim
$w_{1}\in RH_{r}^{+}$ for all $1<r<\frac{C}{C-1}$ with $C=\max \{A_{1}^{+}(w_{1}), A_{1}^{-}(w_{1})\}>1$. In fact, for fixed interval
$I=(a,b)$. We consider the truncation of $w$ at height $H$ defined by $w_{H}=\min\{w_{1},H\}$, which also satisfies $A_{1}^{+}$ with a constant $C_{H}\leq C$. We can therefore obtain that if $\lambda_{I}=M(w_{H}\chi_{I})(b)$ and $S_{\lambda}=\{x\in I:w_{H}(x)>\lambda\}$ then the following statement holds:
$$
\int_{S_{\lambda}}w_{H}(x)dx\leq C_{H}\lambda |S_{\lambda}|,\,\,\lambda\geq \lambda_{I}. \eqno(2.2)
$$
Indeed, it is straightforward if $S_{\lambda}=I$, since
$$w_{H}(S_{\lambda})=\int_{a}^{b}w_{H}(x)dx\leq \lambda_{I}(b-a)\leq C_{H}\lambda |S_{\lambda}|.$$
So we only to consider $S_{\lambda}\neq I$, let us fix $\varepsilon>0$ and an open set $O$ such that $S_{\lambda}\subset O\subset I$ and $|O|\leq \varepsilon+|S_{\lambda}|$. Let $O_{i}=(c,d)$ which is connected. There are two cases:
$a\leq c<d<b$ and $a\leq c<d=b$. In the first case $d$ is not contained in $S_{\lambda}$, and recall the definition of $S_{\lambda}$, $w_{1}^{+}$, we have $\int_{c}^{d}w_{H}(x)dx\leq C_{H}\lambda(d-c)$ while the second case handled as the case $S_{\lambda}= I$, since $\int_{c}^{d}w_{H}(x)dx\leq C(b-c)$. Thus $w_{H}(O_{i})\leq C_{H}\lambda|O_{i}|$. Adding up with $i$ we get $w_{H}(S_{\lambda}\leq w_{H}(O)\leq C_{H}\lambda|O_{i}|\leq C_{H}\lambda(\varepsilon+|S_{\lambda}|)$ and we obtain (2.2). We fix $\theta>-1$, multiply both sides of (2.2) by $\lambda^{\theta}$ and integrate from $\lambda_{I}$ to infinity we have
$$
\frac{1}{\theta+1}\int_{I}\left(w_{H}^{\theta+2}-\lambda_{I}^{\theta+1}\right)(x)dx\leq \frac{C_{H}}{\theta+2}\int_{I}w_{H}^{\theta+2}(x)dx.
$$
Now if $r=\theta+2<\frac{C_{H}}{C_{H}-1}$, then $\frac{1}{\theta+1}-\frac{C_{H}}{\theta+2}>0$, which implies
$$
\int_{I}w_{H}^{r}(x)dx\leq C_{H}\lambda_{I}^{r-1}\int_{I}w_{H}(x)dx=C_{H}\left(M(w_{H}\chi_{I})(b)\right)^{r-1}\int_{I}w_{H}(x)dx.
$$

Since $C_{H}\leq C$ implies $\frac{C_{H}}{C_{H}-1}\geq \frac{C}{C-1}$, therefore if $r\leq \frac{C}{C-1}$, then
$$\begin{array}{rl}
\displaystyle \int_{I}w_{H}^{r}(x)dx&=\displaystyle C_{H}\left(M(w_{H}\chi_{I})(b)\right)^{r-1}\int_{a}^{b}w_{H}(x)dx\\&=\displaystyle C\left(M(w_{1}\chi_{(a,b)})(b)\right)^{r-1}\int_{a}^{b}w_{H}(x)dx
\end{array}$$
So $w_{1}\in RH_{r}^{+}$ follows from the the monotone convergence theorem.
Since $w_{2}\in A_{1}^{-}$, we claim
$w_{2}^{1-p}\in RH_{\infty}^{+}$. In fact, for any interval $I=(a,b)$, by H\"{o}lder's inequality we have
$$
\left(\frac{1}{|I|}\int_{I}w_{2}(x)dx\right)^{1-p}\leq \frac{1}{|I|}\int_{I}w_{2}(x)^{1-p}dx,
$$
 and recall the $A_{1}^{-}$ condition, for almost every $x\in I^{-}=(2a-b,a)$, we have that
 $$Cw_{2}\geq \frac{1}{|I|}\int_{I}w_{2}(x)dx,$$ thus
$$\begin{array}{rl}
\displaystyle w_{2}(x)^{1-p}&\leq\displaystyle C\left(\frac{1}{|I|}\int_{I}w_{2}(x)dx\right)^{1-p}\\
&\leq\displaystyle C\frac{1}{|I|}\int_{I}w_{2}(x)^{1-p}dx\\
&\leq\displaystyle C\frac{1}{b-x}\int_{x}^{b}w_{2}(x)^{1-p}dx,
\end{array}$$
which implies our claim. Hence
\begin{eqnarray*}
\frac{1}{|I|}\int_{I}w^{r}
&\leq& \frac{1}{|I|}\int_{I}w_{1}^{r}\sup_{I}\left(w_{2}^{-r(p-1)}\right)\\
&\leq& C\left(\frac{1}{I_{1}}\int_{I_{1}}w_{1}\right)^{r}\left(\frac{1}{I_{1}}\int_{I_{1}}w_{2}^{1-p}\right)^{r}\\
&\leq& C\left(\inf_{I_{1}}w_{1}\right)^{r}\left(\sup_{I_{1}}w_{2}^{1-p}\right)^{r}\\
&\leq& C\left(\inf w_{1}\right)^{r}\left(\frac{1}{I_{2}}\int_{I_{2}}w_{2}^{1-p}\right)^{r}\\
&\leq& C\left(\frac{1}{I_{2}}\int_{I_{2}}w\right)^{r},
\end{eqnarray*}
where $I_{1}=(b,2b-a)$, $I_{2}=(2b-a,3b-2a)$. By Lemma 2.5, we conclude $w\in RH_{r}^{+}$. By Lemma 2.1, we have $w^{1-p'}\in RH_{r}^{-}$ for all $1<r<\frac{C}{C-1}$.

Let us fix $a<d$ and choose $b,c$ such that $b-a=d-c=\frac{1}{4}(d-a)$ (e.g we choose $b=\frac{d+3a}{4},c=\frac{3d+a}{4}$). Following from the five points $a, b,\frac{b+c}{2}, c, d$, we have four intervals, namely
$$
I_{1}=(a,b), I_{2}=\left(b,\frac{b+c}{2}\right), I_{3}=\left(\frac{b+c}{2},c\right),  I_{4}=(c,d).
$$
By Lemma 2.5, we have
$$
\frac{1}{|I_{1}|}\int_{I}w^{r}\left(\frac{1}{|I_{4}|}\int_{I_{4}}w^{r(1-p')}\right)^{p-1}\leq \left(\frac{1}{|I_{2}|}\int_{I_{2}}w\right)^{r}\left(\frac{1}{|I_{3}|}\int_{I_{3}}w^{(1-p')}\right)^{r(p-1)}\leq C^{r},
$$
thus $w^{r}\in A_{p}^{+}$ which follows from Lemma 2.6. If we choose $0<\varepsilon=r-1<\frac{1}{C-1}$, then we complete the proof of Lemma 2.7.
\end{proof}

To prove Theorem 1.1, we still need a celebrated interpolation theorem of operators with change of measures:
\begin{lemma}$[13]$\,\,
Suppose that $u_{0},v_{0},u_{1},v_{1}$ are positive weight functions and $1<p_{0},p_{1}<\infty$. Assume sublinear operator $S$ satisfies:
$$\|Sf\|_{L^{p_{0}}(u_{0})}\leq C_{0}\|f\|_{L^{p_{0}}(v_{0})},$$
and
$$\|Sf\|_{L^{p_{1}}(u_{1})}\leq C_{1}\|f\|_{L^{p_{1}}(v_{1})}.$$
Then
$$\|Sf\|_{L^{p}(u)}\leq C\|f\|_{L^{p}(v)}$$
holds for any $0<\theta<1$ and $\frac{1}{p}=\frac{\theta}{p_{0}}+\frac{1-\theta}{p_{1}}$, where
$u=u_{0}^{\frac{p\theta}{p_{0}}}u_{1}^{\frac{p(1-\theta)}{p_{1}}}$, $v=v_{0}^{\frac{p\theta}{p_{0}}}v_{1}^{\frac{p(1-\theta)}{p_{1}}}$ and $C\leq C_{0}^{\theta}C_{1}^{1-\theta}$.
\end{lemma}

Lemma 2.7 and Lemma 2.8 are the mains tools in proving of Theorem 1.1.

\section{Proof of Theorem 1.1}
\begin{proof} (1)\,\, Suppose $P(x,y)$ is a real polynomial with degree $k$ in $x$ and degree $l$ in $y$. We shall carry out the argument by induction. First, we assume the conclusion of Theorem 1.1 is valid for all polynomials which are the sums of monomials of degree less than $k$ in $x$ times
monomials of any degree in $y$, together with monomials which are of degree $k$ in $x$  times monomials which are of degree less than $l$ in $y$ .Thus $P(x,y)$ can be written as
$$
P(x,y)=a_{kl}x^{k}y^{l}+R(x,y).
$$
where
$$
R(x,y)=\sum_{\alpha<k,\beta\leq l}a_{\alpha\beta}x^{\alpha}y^{\beta}+\sum_{\beta< l}a_{k\beta}x^{k}y^{\beta}.
$$
satisfying the above induction assumption.

For $kl=0$, the conclusion of Theorem 1.1 holds by the aid of weighted theory of one-sided Calderon-Zygumund operators. Let us now prove that the conclusion of Theorem 1.1 holds for arbitrary $k$ and $l$ by induction. Without loss of generality, we may assume $k>0,l>0$ and $|a_{kl}|\neq 0$ (for if $|a_{kl}|= 0$, Theorem 1.1 holds by the induction assumption).

{\it Case} 1. $|a_{kl}|=1.$

Write
$$\begin{array}{rl}
\displaystyle T^{+}f(x)&=\displaystyle \int_{x}^{1+x}e^{iP(x,y)}K(x-y)f(y)dy+\sum_{j=1}^{\infty}\int_{2^{j-1}+x}^{2^{j}+x}e^{iP(x,y)}K(x-y)f(y)dy\\
&=:\displaystyle T_{0}^{+}f(x)+\sum_{j=1}^{\infty}T_{j}^{+}f(x).
\end{array}$$
Take any $h\in \mathbb{R^{+}}$, and write
$$
P(x,y)=a_{kl}(x-h)^{k}(y-h)^{l}+R(x,y,h),
$$
where the polynomial $R(x,y,h)$ satisfies the induction assumption, and the coefficients of $R(x,y,h)$ depend on $h$.

$1^{\circ}$ {\it Estimates for} $T_{0}^{+}$.

We have
$$\begin{array}{rl}
\displaystyle T_{0}^{+}f(x)&=\displaystyle \int_{x}^{1+x}e^{i(R(x,y,h)+a_{kl}(y-h)^{k+l})}K(x-y)f(y)dy\\
&\,\,+\displaystyle \int_{x}^{1+x}\left\{e^{iP(x,y)}-e^{i(R(x,y,h)+a_{kl}(y-h)^{k+l})}\right\}K(x-y)f(y)dy\\&=:\displaystyle T_{01}^{+}f(x)+T_{02}^{+}f(x).
\end{array}$$
Now we split $f$ into three parts as follows
$$\begin{array}{rl}
\displaystyle f(y)&=\displaystyle f(y)\chi_{\{|y-h|<\frac{1}{2}\}}(y)+f(y)\chi_{\{\frac{1}{2}\leq|y-h|<\frac{5}{4}\}}(y)+f(y)\chi_{\{|y-h|\geq\frac{5}{4}\}}(y)\\
&=:f_{1}(y)+f_{2}(y)+f_{3}(y).
\end{array}$$

It is easy to see that when $|x-h|<\frac{1}{4}$, we have
$$
T_{01}^{+}f_{1}(x)=\int e^{i(R(x,y,h)+a_{kl}(y-h)^{k+l})}K(x-y)f_{1}(y)dy.
$$
Thus, it follows from the induction assumption that
$$
\int_{|x-h|<\frac{1}{4}}|T_{01}^{+}f_{1}(x)|^{p}w(x)dx\leq C\int_{|y-h|<\frac{1}{2}} |f(y)|^{p}w(y)dy,\eqno(3.1)
$$
where $C$ is independent of $h$ and the coefficients of $P(x,y)$.\\

Notice that if $|x-h|<\frac{1}{4},\frac{1}{2}\leq|y-h|<\frac{5}{4}$, then $y-x>\frac{1}{4}$. Thus
$$
|T_{01}^{+}f_{2}(x)|\leq C\int_{x+\frac{1}{4}}^{x+1}|K(x-y)f_{2}(y)|dy\leq CM^{+}(f_{2})(x).
$$
So we have
$$
\int_{|x-h|<\frac{1}{4}}|T_{01}^{+}f_{2}(x)|^{p}w(x)dx\leq C\int_{|y-h|<\frac{5}{4}} |f(y)|^{p}w(y)dy,\eqno(3.2)
$$
where $C$ is independent of $h$ and the coefficients of $P(x,y)$.

Again notice that if $|x-h|<\frac{1}{4},|y-h|\geq\frac{5}{4}$, then $y-x>1$,thus
$$
T_{01}^{+}f_{3}(x)=0.\eqno(3.3)
$$

Combining (3.1), (3.2) and (3.3), we get
$$
\int_{|x-h|<\frac{1}{4}}|T_{01}^{+}f(x)|^{p}w(x)dx\leq C\int_{|y-h|<\frac{5}{4}} |f(y)|^{p}w(y)dy,\eqno(3.4)
$$
where $C$ is independent of $h$ and the coefficients of $P(x,y)$.\\

Evidently, if $|x-h|<\frac{1}{4},0<y-x<1$, then
$$
|e^{iP(x,y)}-e^{i(R(x,y,h)+a_{kl}(y-h)^{k+l})}|\leq |a_{kl}||x-y|=C(y-x).
$$
Therefore, when $|x-h|<\frac{1}{4}$, we have
$$
|T_{02}^{+}f(x)|\leq C\int_{x}^{x+1}|f(y)|dx\leq CM^{+}(f(\cdot)\chi_{B(h,\frac{5}{4})}(\cdot))(x).
$$
It follows that
$$
\int_{|x-h|<\frac{1}{4}}|T_{02}^{+}f(x)|^{p}w(x)dx\leq C\int_{|y-h|<\frac{5}{4}} |f(y)|^{p}w(y)dy,\eqno(3.5)
$$
where $C$ is independent of $h$ and the coefficients of $P(x,y)$.

From (3.4) and (3.5), it follows that the inequality
$$
\int_{|x-h|<\frac{1}{4}}|T_{0}^{+}f(x)|^{p}w(x)dx\leq C\int_{|y-h|<\frac{5}{4}} |f(y)|^{p}w(y)dy,
$$
holds uniformly in $h\in \mathbb{R}^{+}$, which implies
$$
\|T_{0}^{+}f\|_{L^{p}(w)}\leq C\|f\|_{L^{p}(w)},\eqno(3.6)
$$
where $C$ is independent of the coefficients of $P(x,y)$, and $w\in A_{p}^{+}$.

$2^{\circ}$  {\it Estimates for} $T_{j}^{+}f$.

For $j\geq 1$, we have
$$
|T_{j}^{+}f(x)|\leq \int_{2^{j-1}+x}^{2^{j}+x}\frac{|f(y)|}{|x-y|}dy\leq CM^{+}(f)(x),
$$
where $C$ is independent of $j$. By lemma 2.7, we know that there exists $\varepsilon>0$, such that $w^{1+\varepsilon}\in A_{p}^{+}$. Thus we have
$$
\|T_{j}^{+}f\|_{L^{p}(w^{1+\varepsilon})}\leq C\|f\|_{L^{p}(w^{1+\varepsilon})},\eqno(3.7)
$$
where $C$ is independent of $j$. On the other hand, by means of the methods in [5] and [10], we get
$$
\|T_{j}^{+}f\|_{L^{p}}\leq C2^{-j\delta}\|f\|_{L^{p}},\eqno(3.8)
$$
where $C$ is dependents only on the total degree of $P(x,y)$, and $\delta>0$. From (3.7) , (3.8) and Lemma 2.8, it follows that
$$\|T_{j}^{+}f\|_{L^{p}(w)}\leq C2^{-j\theta\delta}\|f\|_{L^{p}(w)},\eqno(3.9)$$
where $0<\theta<1$, $\theta$ is independent of $j$, and $C$ depends only on the total degree of $P(x,y)$.

Now (3.6) and (3.9) imply
$$
\|T^{+}f\|_{L^{p}(w)}\leq C\|f\|_{L^{p}(w)},
$$
where $C$ depends only on the total degree of $P(x,y)$, and $w\in A_{p}^{+}$.

{\it Case} 2. $|a_{kl}|\neq 1$.

Write $\lambda=|a_{kl}|^{\frac{1}{k+l}}$, and
$$
P(x,y)=\lambda^{-(k+l)}a_{kl}(\lambda x)^{k}(\lambda y)^{l}+R(\frac{\lambda x}{\lambda},\frac{\lambda y}{\lambda})=Q(\lambda x,\lambda y).
$$
Thus we have
$$\begin{array}{rl}
\displaystyle T^{+}f(x)&=\displaystyle \mathrm{p.v.}\int e^{iQ(\lambda x, \lambda y)}K(x,y)f(y)dy\\
&=\displaystyle \mathrm{p.v.}\int e^{iQ(\lambda x,y)}K(\frac{\lambda x}{\lambda},\frac{y}{\lambda})f(\frac{y}{\lambda})\lambda^{-1}dy\\&=\displaystyle \lambda^{-1}T^{+}_{\lambda}(f(\frac{\cdot}{\lambda}))(\lambda x),
\end{array}$$
where $K_{\lambda}(x,y)=K(\frac{x}{\lambda},\frac{y}{\lambda})$ and
$$
T^{+}_{\lambda}f(x)=\mathrm{p.v.}\int e^{iQ(x,y)}K_{\lambda}(x,y)f(y)dy.
$$
It is esay to see that $K_{\lambda}$ satisfies (1.1), (1.2), and the operator $f\mapsto p.v.\int K_{\lambda}(x,y)f(y)dy$ is of type $(L^{2},L^{2})$. Therefore, from the conclusion in {\it Case} 1, we obtain
$$
\|T^{+}_{\lambda}f\|_{L^{p}(w)}\leq C\|f\|_{L^{p}(w)}.
$$
where $w\in A_{p}^{+}$ and $C$ depends only on the total degree of $P(x,y)$. Noticing Lemma 2.2, we have
\begin{eqnarray*}
\int |T^{+}f(x)|^{p}w(x)dx
&=&\lambda^{-p}\int \left|T_{\lambda}^{+}f(\frac{\cdot}{\lambda})(\lambda x)\right|^{p}w(x)dx\\
&=&\lambda^{-p-1}\int \left|T_{\lambda}^{+}f(\frac{\cdot}{\lambda})(x)\right|^{p}w(\frac{x}{\lambda})dx\\
&\leq&C\int \left|f(\frac{x}{\lambda})\right|^{p}w(\frac{x}{\lambda})dx\\
&=&C\int |f(x)|^{p}w(x)dx,
\end{eqnarray*}
that is $\|T^{+}f\|_{L^{p}(w)}\leq C\|f\|_{L^{p}(w)}$, where $C$ depends only on the total degree of $P(x,y)$, but not on the coefficients of $P(x,y)$, and $w\in A_{p}^{+}$.

(2)\,\, We omit the details, since they are very similar to those of the proof of (1).
\end{proof}


\end{document}